\documentclass[10pt]{amsart}

\usepackage[notref,notcite]{showkeys}
\usepackage{latexsym,enumerate}
\usepackage{amsmath,amsthm,amsopn,amstext,amscd,amsfonts,amssymb}
\usepackage{color}
\usepackage{fancybox}
\usepackage{epsfig}
\usepackage{float}
\usepackage{ulem} 

\usepackage[active]{srcltx}

\usepackage{color}

\newcommand{\R}{\mathbb{R}}
\newcommand{\N}{\mathbb{N}}

\newcommand{\car}{{\raise0pt\hbox{{\LARGE $\chi$}}}}

\newcommand{\sg}{{\rm \; sign}}

\newcommand{\Div}{\hbox{\rm div\,}}

\newcommand{\e }{\varepsilon }

\newcommand{\dis }{{\mathcal D}' }
\newcommand{\z }{{\bf z}}
\newcommand{\DM }{\mathcal{DM}^\infty }

\newtheorem{Theorem}{Theorem}[section]
\newtheorem{Corollary}[Theorem]{Corollary}
\newtheorem{Definition}[Theorem]{Definition}

\newtheorem{Proposition}[Theorem]{Proposition}

\newtheorem{remark}[Theorem]{Remark}
\newtheorem{Example}[Theorem]{Example}
\newcommand{\res}{\!\!\mathop{\hbox{
                                \vrule height 7pt width .5pt depth 0pt
                                \vrule height .5pt width 6pt depth 0pt}}
                                \nolimits}
\newcommand{\norma}[2]{\|#1\|_{\lower 4pt \hbox{$\scriptstyle #2$}}}

\usepackage{multirow}
\usepackage{array}

\newcommand{\h}{{\mathcal H}} 

\begin{document}

\title[An equation involving the $1$--Laplacian and $L^1$--data]{Existence and comparison results for an elliptic equation involving the $1$--Laplacian and $L^1$--data}

\author[M. Latorre and  S. Segura de Le\'on]
{Marta Latorre and  Sergio Segura de Le\'on}

\address{M. Latorre: Departament d'An\`{a}lisi Matem\`atica,
Universitat de Val\`encia,
Dr. Moliner 50, 46100 Burjassot, Spain.
{\it E-mail address:} {\tt marta.latorre@uv.es }}

\address{S. Segura de Le\'on: Departament d'An\`{a}lisi Matem\`atica,
Universitat de Val\`encia,
Dr. Moliner 50, 46100 Burjassot, Spain.
{\it E-mail address:}  {\tt sergio.segura@uv.es }}

\thanks{}


\bigskip
\begin{abstract}
This paper is devoted to analyse the Dirichlet problem for a nonlinear elliptic equation involving the $1$--Laplacian and a total variation term, that is, the inhomogeneous case of the equation arising in the level set formulation of the inverse mean curvature flow.
We study this problem in an open bounded set with Lipschitz boundary.

We prove an existence result and a comparison principle for non--negative $L^1$--data. Moreover, we search the summability that the solution reaches when more regular $L^p$--data, with $1<p<N$, are considered and we give evidence that this summability is optimal.

To prove these results, we apply the theory of $L^\infty$--divergence--measure fields which goes back to Anzellotti (1983). The main difficulties of the proofs come from the absence of a definition for the pairing of a general $L^\infty$--divergence--measure field and the gradient of an unbounded $BV$--function.

\end{abstract}

\maketitle

\tableofcontents

\section{Introduction}

Our aim in this paper is to analyze the following Dirichlet problem:
\begin{equation}\label{prob-prin}
\left\{\begin{array}{ll}
\displaystyle -\Div\left(\frac{Du}{|Du|}\right)+|Du|=f(x)&\hbox{ in }\,\Omega\,,\\[3mm]
u=0 &\hbox{ on }\,\partial\Omega\,,
\end{array}\right.
\end{equation}
where $\Omega$ is a bounded open subset of $\R^N$ with Lipschitz boundary $\partial\Omega$ and $f$ is a non--negative function belonging to $L^1(\Omega)$. As usual when the $1$--Laplacian operator is considered, the natural energy space to study this problem is $BV(\Omega)$, that is, the space of all functions of bounded variation.

The homogeneous problem, in an unbounded domain, arises in the level set formulation of the inverse mean curvature flow, namely,
\begin{equation}\label{IMCF}
\left\{\begin{array}{ll}
\displaystyle -\Div\left(\frac{Du}{|Du|}\right)+|Du|=0&\hbox{ in }\,\Omega\,,\\[3mm]
u=0 &\hbox{ on }\,\partial\Omega\,,\\[3mm]
u(x)\to \infty &\hbox{ as }\, |x| \to \infty\,.
\end{array}\right.
\end{equation}
The inverse mean curvature flow is a one--parameter family of hypersurfaces $\{\Gamma_t\}_{t \geq 0}$ whose normal velocity $V_n(t)$ at each time $t$ equals the inverse of its mean curvature $H(t)$. Given $\Gamma_0$, the problem is to find $F: \Gamma_0 \times [0, T] \rightarrow \R^N$ such that
\begin{equation}\label{e1imc}
\frac{\partial F}{\partial t} = \frac{\nu}{H}\,, \quad \quad t \geq 0\,,
\end{equation}
where $\nu(t)$ denotes the unit outward normal to $\Gamma_t=F(\Gamma_0, t)$. The level set formulation \eqref{IMCF} was introduced in \cite{Huisken} (see also \cite{Huisken2, Moser}); observe that $\Div\left(\frac{Du}{|Du|}\right)$ gives the mean curvature and $|Du|$ yields the inverse of the speed. In the case that $\Omega$ includes a bounded connected component, it produces a sudden phenomenon called fattening by which this component disappear instantaneously.

If a non--negative source is considered, as in \eqref{prob-prin}, then \eqref{e1imc} becomes
\begin{equation*}
\frac{\partial F}{\partial t} = \frac{\nu}{H+\hbox{source}}\le \frac{\nu}{H}, \quad \quad t \geq 0,
\end{equation*}
so that the datum damps the flux. This inhomogeneous inverse mean curvature flow was studied in \cite{MS0}.

Although the homogeneous problem is not interesting in bounded domains because it leads to the trivial solution, this does not occur in the non--homogeneous case since the source can override the fattening phenomenon (at least when $f$ is not very small). Problem \eqref{prob-prin} in bounded domains has been considered in \cite{MS} for data $f\in L^p(\Omega)$, with $p>N$, seeking bounded solutions, and in \cite{LS} when data belong to the Marcinkiewicz space $L^{N,\infty}(\Omega)$, looking for unbounded variational solutions. Existence and uniqueness results have been obtained in both papers for any given non--negative datum.

It is worth mentioning that the gradient term is essential to get existence and uniqueness results. In \cite{K} (see also \cite{CT, MST1} for more general data) it is shown that there exist solutions to problem
\begin{equation*}
\left\{\begin{array}{ll}
\displaystyle -\Div\left(\frac{Du}{|Du|}\right)=f(x)&\hbox{ in }\,\Omega\,,\\[3mm]
u=0 &\hbox{ on }\,\partial\Omega\,,
\end{array}\right.
\end{equation*}
only when data are small enough. On the other hand, uniqueness cannot be expected since if $u$ is a solution and $g$ is a real increasing smooth function, then $v=g(u)$ should be a solution as well. Therefore, the total variation term has a regularizing effect.

Our purpose is to go a step further and study problem \eqref{prob-prin} when data are merely integrable functions. This kind of non--variational problems has extensively been studied for problems involving the $p$--Laplacian ($1<p\le N$). In this framework, there are two different formulations: that of entropy solution introduced in \cite{B-V} (see also \cite{BGO}) and that of renormalized solution, for which we refer to \cite{DMOP}. Both approaches systematically use truncations of solutions. In \cite{ABCM}, in the framework of the $1$--Laplacian, the authors also introduce a notion of solution by means of truncations. We follow the same concept, but adapted to our situation. Indeed, since the regularizing effect of the total variation yields $u\big|_{\partial\Omega}=0$, the boundary condition holds in the sense of traces.

Another feature deriving from the regularizing effect is $u\in BV(\Omega)$ without jump part. Nevertheless, this fact does not allow us to define (following Anzellotti, see \cite{An}) the pairing of a general $L^\infty$--divergence--measure vector field $\z$ and the solution $u$. Hence, truncations must remain in the definition of solution. Instead of products of the form $(\z,Du)$, we have to handle with products such as $(\z,De^{-u})$ and $(e^{-u}\z,Du)$. Beyond these kind of technical complication, the existence theorem holds as it was expected, and we will only make explicit those parts of the proof which are different of that for regular data in \cite[Theorem 4.3]{LS}. Much more interesting is the comparison principle. We point out that, even in the context of bounded solutions, its proof is new and simpler than that of the uniqueness result in \cite{MS}. We also investigate solutions when data belong to $L^p(\Omega)$, with $1<p<N$, finding that the solution lies in $L^{\frac{Np}{N-p}}(\Omega)$. Note that Lebesgue spaces continuously adjust with the known cases $p=1$ (in which $u\in BV(\Omega)\subset L^{\frac N{N-1}}(\Omega)$) and $p=N$ (see \cite[Proposition 4.7]{LS}).

This paper is organized as follows. In section 2 we introduce some definitions and notation and we also give some preliminaries results that we will need. Among these results, we foreground Proposition 2.4 for which we supply a new proof. This proposition is essential to deal with pairings involving functions of the solution (such as truncations) through Proposition 2.7. Section 3 is devoted to prove the existence result and the comparison principle. In section 4 we show the best summability that the solution can get when data belong to $L^p(\Omega)$, with $1<p<N$. Finally, in the last section  we show examples of radial solutions, which give evidence that the obtained regularity is optimal.

\section*{Acknowledgements}
This research has been partially supported  by the Spanish Mi\-nis\-te\-rio de Econom\'{\i}a y Competitividad and
FEDER, under project MTM2015--70227--P.
The first author was also supported by Ministerio de Econom\'{\i}a y
Competitividad under grant BES--2013--066655.

The authors would like to thank Salvador Moll for some fruitful discussions concerning this paper.

\section{Preliminaries}

   In this section we will introduce some notation and auxiliary results which will be used
    throughout this paper. In what follows, we will consider $N\ge2$ and, given a set $E$, we will write
   $\h^{N-1}(E)$ to denote its $(N - 1)$--dimensional Hausdorff measure and $|E|$ its Lebesgue measure.

  In this paper, $\Omega$ will always stands for an open subset of
  $\R^N$ with Lipschitz boundary. Thus, an outward normal unit
  vector $\nu(x)$ is defined for $\h^{N-1}$--almost every
  $x\in\partial\Omega$.
 We will make use of the usual Lebesgue and Sobolev
 spaces, denoted by $L^q(\Omega)$  and $W_0^{1,p}(\Omega)$,
 respectively (see for instance \cite{Br} or \cite{Ev}).

We recall that for a Radon measure $\mu$ in $\Omega$ and a Borel set
$A\subseteq\Omega$, the measure $\mu\res A$ is defined by $(\mu\res A)(B)=\mu(A\cap B)$
for any Borel set $B\subseteq\Omega$.

The truncation function will be use throughout this paper. Given
$k>0$, it is defined by
\begin{equation}\label{trun}
    T_k(s)=\min\{|s|, k\}\sg (s)\,,
\end{equation} for all $s\in\R$.

\subsection{Functions of bounded variation}
The natural energy space to study problems involving the $1$--Laplacian is the space of all functions of bounded variation, that is, functions $u\,:\,\Omega\to\R$ belonging to $L^1(\Omega)$ whose derivative in the sense of distributions $Du$ is a Radon measure with finite total variation. This space will be denoted by $BV(\Omega)$.

Let $u \in BV(\Omega)$, we can decompose the Radon measure $Du$ into its absolutely continuous part and its singular part with respect to the Lebesgue measure: $Du=D^au + D^su$. We denote by $S_u$ the set of all $x \in \Omega$ which are not Lebesgue points, that is, $x \not\in S_u$ if there exists $\tilde{u}(x)$ such that
\begin{equation*}
\lim_{\rho \downarrow 0} \frac{1}{|B_{\rho}(x)|} \int_{B_{\rho}(x)} \vert u(y) - \tilde{u}(x) \vert \, dy = 0 \,.
\end{equation*}

We say that $x \in \Omega$ is an approximate jump point of $u$, denoted by $x \in J_u$, if there exist two real numbers $u^+(x)>u^-(x)$ and $\nu_u(x)$ with $|\nu_u(x)|=1$ such that
\begin{equation*}
\lim_{\rho \downarrow 0} \frac{1}{|B_{\rho}^+(x,\nu_u(x))|} \int_{B_{\rho}^+(x,\nu_u(x))} \vert u(y) - u^+(x) \vert \, dy = 0 \,,
\end{equation*}
\begin{equation*}
\lim_{\rho \downarrow 0} \frac{1}{|B_{\rho}^-(x,\nu_u(x))|} \int_{B_{\rho}^-(x,\nu_u(x))} \vert u(y) - u^-(x) \vert \, dy = 0 \,,
\end{equation*}
where
\begin{equation*}
B_{\rho}^+(x,\nu_u(x)) = \{ y \in B_{\rho}(x) \ | \ \langle y - x, \nu_u(x) \rangle >0 \}
\end{equation*}
and
\begin{equation*}
B_{\rho}^-(x,\nu_u(x)) = \{ y \in B_{\rho}(x) \ | \ \langle y - x, \nu_u(x) \rangle <0 \} \,.
\end{equation*}
We know that $S_u$ is countably $\h^{N-1}$--rectifiable and $\h^{N-1}(S_u \backslash J_u) = 0$ by the Federer--Vol'pert Theorem (see \cite[Theorem 3.78]{AFP}). Moreover, we also know
\begin{equation*}
Du \res J_u = (u^+ - u^-) \nu_u \h^{N-1} \res J_u \,.
\end{equation*}
Using $S_u$ and $J_u$, we can split $D^su$ in its jump part $D^j u$ and its Cantor part $D^c u$, defined by
\begin{equation*}
D^ju = D^su \res J_u  \qquad  {\rm and} \qquad D^c u = D^su \res (\Omega \backslash S_u) \,.
\end{equation*}
Then, we have
\begin{equation*}
D^j u = (u^+ - u^-) \nu_u \h^{N-1} \res J_u \,.
\end{equation*}
In addition, if $x \in J_u$, then $\nu_u(x) = \frac{Du}{| D u |}(x)$ where $\frac{Du}{| D u |}$ is the Radon--Nikod\'ym derivative of $Du$ with respect to its total variation $| D u |$.

The precise representative $u^* : \Omega \setminus(S_u \setminus J_u) \rightarrow \R$ of
$u$ is defined by
\begin{equation*}
u^*(x) = \left\{
\begin{array}{lcc}
\tilde{u}(x) & \mbox{if} & x \in \Omega \setminus S_u \,,\\
\dfrac{u^-(x) + u^+(x)}{2} & \mbox{if} & x \in J_u \,.
\end{array}
\right.
\end{equation*}
For the sake of simplicity, most of the time we will denote both function and its precise representative by $u$.

We will use the Chain Rule, but only when $u $ is a bounded variation function without jump part.
\begin{Proposition}\label{Chain-Rule}
Let $u \in BV(\Omega)$ with $D^ju=0$ and let $f$ be a Lipschitz function in $\Omega$. Then, $v=f \circ u$ belongs to $BV(\Omega)$ and $Dv=f^\prime(u) Du$, so that $D^j v=0$.
\end{Proposition}

For further information about bounded variation functions we refer to \cite{AFP}, \cite{EG} and \cite{Zi}.

\subsection{$L^\infty$--divergence--measure fields}
We will denote by $\DM(\Omega)$ the set of all vector fields $\z \in L^\infty(\Omega;\R^N)$ such that $\Div \z$ is a Radon measure in $\Omega$ with finite total variation. Following \cite{ABCM}, we will use these vector fields to give a sense to $\displaystyle \frac{Du}{|Du|}$  in our equation, even if $Du$ is a Radon measure and, moreover, if it vanishes in a zone of the domain. More concretely, we seek for a vector field $\z\in L^\infty(\Omega;\R^N)$ satisfying $\|\z\|_\infty\le1$ and $(\z,DT_k(u))=|DT_k(u)|$ for all $k>0$.

Let $\z \in \DM(\Omega)$ and $w \in BV(\Omega)\cap C(\Omega) \cap L^\infty(\Omega)$; for every $\varphi \in C^\infty_0 (\Omega)$ we define the functional
\begin{equation*}
\langle (\z, Dw),\varphi \rangle = - \int_\Omega w \,\varphi \, \Div \z - \int_\Omega w \, \z \cdot \nabla \varphi \,dx \,.
\end{equation*}
It was proved in \cite{An} that this distribution has order $0$ since satisfies
\[
|\langle (\z, Dw),\varphi\rangle|\le\|\varphi\|_\infty\|\z\|_\infty \int_\Omega|Dw|\,.
\]
Thus, it is actually a Radon measure with finite total variation and the following inequality holds
\begin{equation}\label{des-An}
  |(\z, Dw)|\le \|\z\|_\infty |Dw|
\end{equation}
as measures in $\Omega$. In particular, the Radon measure $(\z, Dw)$ is absolutely continuous with respect to $|Dw|$.
Denoting by $$\theta(\z, Dw, \cdot) : \Omega \rightarrow \R$$ the
Radon--Nikod\'ym derivative of $(\z,Dw)$ with respect to $|Dw|$,
it follows that
\begin{equation*}
\int_B (\z, Dw) = \int_B \theta(\z, Dw, x) \, |Dw| \qquad \mbox{ for all Borel
sets} \quad B \subset \Omega \,,
\end{equation*}
and
\begin{equation*}
\Vert \theta(\z, Dw, \cdot)\Vert_{L^{\infty}(\Omega, \vert Dw \vert)} \leq \Vert \z \Vert_{\infty}\,.
\end{equation*}
Moreover, if $f : \R \rightarrow \R$ \ is a Lipschitz continuous increasing function, then
\begin{equation}\label{E1paring12}
\theta(\z, D(f \circ w),x) = \theta (\z, Dw, x) \qquad  |Dw|\mbox{--a.e.} \quad {\rm in} \quad \Omega \,.
\end{equation}

The Anzellotti theory also provides the definition of a weak trace on $\partial
\Omega$ to the normal component of any vector field $\z\in \DM$, denoted by $[z,\nu]$. This weak trace satisfies $\|[\z,\nu]\|_\infty\le \|\z\|_\infty$. Relating the pairing $(\z, Dw)$ and the weak trace $[z,\nu]$ a Green's formula holds.

\begin{Theorem}\label{Green}
If $\z\in \DM$ and $w\in BV(\Omega)\cap C(\Omega) \cap L^\infty(\Omega)$, then we have
\[
\int_\Omega w\,\Div\z +\int_\Omega (\z, Dw)=\int_{\partial\Omega}w\,[\z,\nu]\, d\h^{N-1}\,.
\]
\end{Theorem}

\bigskip

As mentioned, for a general $\z \in \DM(\Omega)$, Anzellotti's theory assumes that $w \in BV(\Omega)\cap C(\Omega) \cap L^\infty(\Omega)$ in order to define $(\z, Dw)$ and to prove a Green's formula. This theory was generalized to consider $w \in BV(\Omega)\cap L^\infty(\Omega)$ in \cite{CF} using a different approach, and in \cite{C} and \cite{MST2} following the same definitions of Anzellotti. Indeed, given $\z \in \DM(\Omega)$ and $w \in BV(\Omega)\cap L^\infty(\Omega)$; for every $\varphi \in C^\infty_0 (\Omega)$ we may define the functional
\begin{equation*}
\langle (\z, Dw),\varphi \rangle = - \int_\Omega w^* \,\varphi \, \Div \z - \int_\Omega w \, \z \cdot \nabla \varphi \,dx \,.
\end{equation*}
We explicitly mention that the precise representative $w^*$ is summable with respect to $\Div\z$ and that this definition depends on the chosen representative of the function.

Now, we present some results which we use several times in the sequel. Next proposition was proved in \cite{MS}.

\begin{Proposition}\label{Prop2.3}
Let $\z \in \DM(\Omega)$ and let $u , w \in BV(\Omega)\cap L^\infty (\Omega)$ be functions such that $D^ju=D^jw=0$. Then
\begin{equation*}
(w\,\z,Du)= w^*(\z,Du) \quad \mbox{as Radon measures in } \; \Omega \,.
\end{equation*}
\end{Proposition}

In principle, it is not clear that (\ref{E1paring12}) holds in the case that $\z \in \DM(\Omega)$ and
$u \in BV(\Omega) \cap L^{\infty}(\Omega)$. However, we will see that (\ref{E1paring12}) holds if we assume the jump part $D^j u$ vanishes.
This result was proved in \cite{MS} but an extra hypothesis is needed in the proof, namely, the set of discontinuities of $u$ is $\h^{N-1}$--null.
We next prove this result under the general assumption $D^ju=0$.
Following Anzellotti, the main ingredient  to prove the above formula is a
``slicing" result that links the measure $(\z, Du)$ with the
measures $(\z, D\car_{E_{u,t}})$, where $E_{u,t}:= \{ x \in \Omega \ : \ u(x) > t \}$.

\begin{Proposition}\label{rebanada} Let $\z \in \DM(\Omega)$ and consider
 $u \in BV(\Omega) \cap L^{\infty}(\Omega)$ with $D^j u=0$. Let $E_{u,t}:= \{ x \in \Omega \ : \ u(x) > t \}$. Then
 for
all $\varphi \in C^{\infty}_{0}(\Omega)$, the function $t \mapsto
\langle (\z, D\car_{E_{u,t}}), \varphi \rangle$ is $\mathcal L^1$--measurable and
\begin{equation}\label{e2mejora}
\langle (\z, Du), \varphi \rangle = \int_{- \infty}^{+ \infty} \langle (\z, D\car_{E_{u,t}}), \varphi \rangle \, dt\,.
\end{equation}
\end{Proposition}

\begin{proof} First we observe that we may assume $u\ge0$; if not, we consider the function $u+\|u\|_\infty$.

We also point out that for every measurable set $E\subset \Omega$ having finite perimeter, the condition $|\Div\z|(\partial^*E)=0$ implies
\[
\car_E\, \Div\z = \car^*_E\, \Div\z\,.
\]

\bigskip
As a consequence, we obtain the following claim:
\\
{\sl If $E\subset \Omega$ is a measurable set with finite perimeter such that $|\Div\z|(\partial^*E)=0$, then
\[
\langle (\z,D\car_E),\varphi\rangle=-\int_E\varphi\, \Div\z-\int_E\z\cdot\nabla \varphi
\]
for all $\varphi\in C_0^\infty(\Omega)$.}

In what follows, recall that $u$ stands for the precise representative of the $BV$--function.
Observe that, thanks to the coarea formula, the level sets $E_{u,t}$ have finite perimeter for $\mathcal L^1$--almost all $t\in\R$. Moreover, since  $D^j u=0$, it follows that
$$\h^{N-1}\left(\partial^* E_{u,t} \cap \partial^* E_{u,s} \right) = 0\quad \hbox{for } s \ne t\,.$$
Then, applying $\vert \Div(\z) \vert \ll \h^{N-1}$ (by \cite[Proposition 3.1]{CF}), we have
\begin{equation*}
|\Div(\z)|\left(\partial^* E_{u,t} \cap \partial^* E_{u,s} \right) = 0 \quad \hbox{ if} \ \  s \ne t\,.
\end{equation*}
Therefore, there exists $A \subset \R$ numerable such that
\begin{equation*}
|\Div(\z)| \left(\partial^* E_{u,t}\right) = 0 \quad \hbox{if} \ \ t \in \R \backslash A\,.
\end{equation*}
In other words, we have seen that $\vert \Div(\z) \vert\left(\partial^* E_{u,t}\right) = 0$ for $ \mathcal{L}^1$--almost all $t>0$.
Thus, our claim implies that if $\varphi\in C_0^\infty(\Omega)$, then
\begin{equation}\label{ec:1}
\langle (\z,D\car_{E_{u,t}}),\varphi\rangle=-\int_{E_{u,t}}\varphi \, \Div\z-\int_{E_{u,t}}\z\cdot\nabla \varphi \,dx\,,\quad \hbox{for }  \mathcal{L}^1\hbox{--almost all } t>0\,.
\end{equation}

Considering $\varphi\in C_0^\infty(\Omega)$, we apply the slicing formula for integrable functions (see, for instance, \cite[Lemma 1.5.1]{Zi}) and \eqref{ec:1} to get that the function
$$
t \mapsto -\int_{E_{u,t}} \varphi\, \Div\z\, dt-\int_{E_{u,t}} \z\cdot\nabla\varphi \,dx
$$
is $\mathcal L^1$--measurable and
\begin{multline*}
\langle (\z,Du),\varphi\rangle= -\int_\Omega u^*\varphi\, \Div\z-\int_\Omega u\,\z\cdot\nabla\varphi
\\= \int_0^\infty\left[-\int_{E_{u,t}} \varphi\, \Div\z\, dt-\int_{E_{u,t}} \z\cdot\nabla\varphi \,dx\right]\, dt
\\=\int_0^\infty\langle(\z,D\car_{E_{u,t}}),\varphi\rangle\, dt \,,
\end{multline*}
as desired.
\end{proof}

\begin{Proposition}\label{rebanada1} Let $\z \in \DM(\Omega)$ and consider
 $u \in BV(\Omega) \cap L^{\infty}(\Omega)$ with $D^j u=0$. Let $E_{u,t}:= \{ x \in \Omega \ : \ u(x) > t \}$. Then
 for
all Borel set $B\subset\Omega$, the function $t \mapsto
\int_B (\z, D\car_{E_{u,t}})$ is $\mathcal L^1$--measurable and
\begin{equation}\label{e3mejora}
\int_B  (\z, Du) = \int_{- \infty}^{+ \infty} \left[\int_B  (\z, D\car_{E_{u,t}})\right] \, dt.
\end{equation}
\end{Proposition}

\begin{proof} Let $S$ denote a countable set in $C_0^\infty(\Omega)$ which is dense with respect to the uniform convergence.
Then, for every $t\in\R$ such that $E_{u,t}$ has finite perimeter and for every $\varphi\in C_0^\infty(\Omega)$ with $\varphi\ge0$, it yields
\[
\langle (\z, D\car_{E_{u,t}})^+,\varphi\rangle =\sup\{\langle (\z, D\car_{E_{u,t}}),\psi\rangle\>:\>\psi\in S\,,\ 0\le\psi\le\varphi\}\,.
\]
Thus the positive part of the measure $t\mapsto \langle (\z, D\car_{E_{u,t}})^+,\varphi\rangle $ defines a $\mathcal L^1$--measurable function since it is the supremum of a countable quantity of $\mathcal L^1$--measurable functions. Recalling the Riesz Representation Theorem, we may go further considering an open set $B\subset\Omega$: it follows from
\[
\int_B(\z, D\car_{E_{u,t}})^+=\sup\{\langle (\z, D\car_{E_{u,t}})^+,\psi\rangle\>:\>\psi\in S\,,\ 0\le\psi\le\car_B\}\,,
\]
that $t\mapsto \int_B(\z, D\car_{E_{u,t}})^+ $ defines a $\mathcal L^1$--measurable function. The regularity of the measures lead to the same conclusion for an arbitrary Borel set. This function is $\mathcal L^1$--summable since
\[
\int_B(\z, D\car_{E_{u,t}})^+\le\int_B|(\z, D\car_{E_{u,t}})|\le\|\z\|_\infty\int_B|D\car_{E_{u,t}}|\,,
\]
for $\mathcal L^1$--almost all $t\in\R$,
and $t\mapsto \int_B|D\car_{E_{u,t}}|$ defines an $\mathcal L^1$--summable function, due to the coarea formula.

On the other hand, a similar argument can be done for the negative part of the measures $(\z, D\car_{E_{u,t}})$, so that $t\mapsto \int_B(\z, D\car_{E_{u,t}})^-$ defines an $\mathcal L^1$--summable function for every Borel set $B\subset\Omega$. As a consequence, $t\mapsto \int_B(\z, D\car_{E_{u,t}})$ defines an $\mathcal L^1$--summable function for every Borel set $B\subset\Omega$.

Finally, consider a distribution $\mu$ defined by
\[
\langle\mu,\varphi\rangle=\langle(\z,Du),\varphi\rangle-\int_{-\infty}^{+\infty}\langle(\z,D\car_{E_{u,t}}),\varphi\rangle\, dt\,.
\]
Proposition \ref{rebanada} implies that $\langle\mu,\varphi\rangle=0$ for all $\varphi\in C_0^\infty(\Omega)$, wherewith $\mu$ is a Radon measure which vanishes identically. Therefore, \eqref{e3mejora} holds true.
\end{proof}

\begin{Corollary}
Let $\z \in \DM(\Omega)$ and consider $u \in BV(\Omega) \cap L^{\infty}(\Omega)$ with $D^j u=0$. Then
\begin{equation}\label{e1mejora}
\theta(\z, Du, x) = \theta(\z, D\car_{E_{u,t}},x) \ \ \vert D\car_{E_{u,t}} \vert\hbox{--a.e.  in }  \Omega \; \hbox{ for} \ \mathcal{L}^1\hbox{--almost all} \ t \in \R\,,
\end{equation}
\end{Corollary}

\begin{proof} Let $a,b\in\R$, with $a<b$ and let $B\subset\Omega$ be a Borel set.
Applying \eqref{e3mejora} to the set $\{x\in\Omega\>:\>a\le u(x)\le b\}\cap B$, we obtain
\begin{equation}\label{basic}
\int_{\{a\le u\le b\}\cap B}(\z,Du)=\int_a^b\left[\int_B(\z,D\car_{E_{u,t}})\right]\, dt\,.
\end{equation}
Now we are analyzing both sides of \eqref{basic}. On the one hand, the coarea formula implies
\begin{multline*}
\int_{\{a\le u\le b\}\cap B}(\z,Du)=\int_{\{a\le u\le b\}\cap B}\theta(\z,Du,x)|Du|
\\=\int_a^b\left[\int_B\theta(\z,Du,x)|D\car_{E_{u,t}}|\right]\, dt\,.
\end{multline*}
On the other,
\[
\int_a^b\left[\int_B(\z,D\car_{E_{u,t}})\right]\, dt=\int_a^b\left[\int_B\theta(\z,D\car_{E_{u,t}},x)|D\car_{E_{u,t}}|\right]\, dt\,.
\]
Hence \eqref{basic} becomes
\[
\int_a^b\left[\int_B\theta(\z,Du,x)|D\car_{E_{u,t}}|\right]\, dt=\int_a^b\left[\int_B\theta(\z,D\car_{E_{u,t}},x)|D\car_{E_{u,t}}|\right]\, dt\,.
\]
It follows that, for $\mathcal L^1$--almost all $t\in\R$,
\[\int_B\theta(\z,Du,x)|D\car_{E_{u,t}}|=\int_B\theta(\z,D\car_{E_{u,t}},x)|D\car_{E_{u,t}}|\]
holds for every Borel set $B$. The desired equality \eqref{e1mejora} is proved.
\end{proof}

\begin{Proposition}\label{Prop2.2} Let $\z \in \DM(\Omega)$ and consider
 $u \in BV(\Omega) \cap L^{\infty}(\Omega)$ with $D^j u=0$. If $f : \R \rightarrow \R$  is a Lipschitz continuous non--decreasing function, then
\begin{equation}\label{E1paring1200}
\theta(\z, D(f \circ u),x) = \theta (\z, Du, x) \quad \vert D(f \circ u) \vert\hbox{--a.e. in }\ \Omega \,.
\end{equation}
\end{Proposition}

\begin{proof} We may follow Anzellotti (see \cite[Proposition 2.8]{An}) for the case of a increasing function. For the general case, consider $f$ non--decreasing and let $\epsilon>0$. Since the function given by
$t \mapsto f(t)+\epsilon t$ is increasing, it follows that
\begin{multline*}
(\z, D(f \circ u))+\epsilon(\z, Du)=(\z, D((f \circ u)+\epsilon u)) \\ = \theta (\z, Du, x)|D((f \circ u)+\epsilon u)|= \theta (\z, Du, x)(f^\prime( u)+\epsilon) |Du|
\end{multline*}
as measures in $\Omega$. Letting $\epsilon\to0$, we deduce
\begin{equation*}
  (\z, D(f \circ u))=\theta (\z, Du, x) |D(f \circ u)|\quad\hbox{as measures in }\Omega\,.
\end{equation*}
Therefore, we have seen that \eqref{E1paring1200} holds.
\end{proof}

\section{Main results}

In this section, we prove our main results, namely the existence theorem and the comparison principle. We begin by stating our concept of solution to problem \eqref{prob-prin}.
The first difficulty we have to deal with is that we are not able to define the distribution $(\z, Du)$ when data are just integrable functions. Following \cite{ABCM}, we will solve this problem introducing truncations in the concept of solution used in \cite{LS}.

\begin{Definition}\label{def-sol}
We say that $u \in BV(\Omega)$ is a solution to problem \eqref{prob-prin} if $D^ju=0$ and there exists a vector field $\z \in \DM(\Omega)$ with $\|\z\|_\infty \le 1$ such that
\begin{equation}\label{cond-ditribucion}
-\Div \z +|Du| =f \,\text{ in }\,\dis (\Omega)\,,
\end{equation}
\begin{equation}\label{cond-medida}
(\z, DT_k(u))=|DT_k(u)| \,\text{ as measures in }\, \Omega \;\; (\mbox{for every } k >0)\,,
\end{equation}
and
\begin{equation}\label{cond-frontera}
u\big|_{\partial \Omega} = 0\,.
\end{equation}
\end{Definition}

\subsection{Existence Theorem}
\begin{Theorem}
Let $\Omega$ be an open and bounded subset of $\R^N$ with Lipschitz boundary and let $f$ be a non--negative function in $L^1(\Omega)$. Then, problem \eqref{prob-prin} has at least one solution.
\end{Theorem}
\begin{proof}
The same proof of \cite[Theorem 4.3]{LS} works with minor modifications. Nevertheless, some remarks are in order.

The first remark is concerning the pairing $(e^{-u}\,\z,Du)$. If $u$ is integrable with respect to the measure $\Div(e^{-u}\z)$ and $\varphi\in C_0^\infty(\Omega)$, then the integrals
\begin{equation*}
 \int_\Omega \varphi \, u \, \Div(e^{-u}\z)\quad \hbox{and}\quad \int_\Omega u\,e^{-u}\z \cdot \nabla \varphi \, dx
\end{equation*}
are both finite; notice that the second integral is bounded due to the inequality $u\,e^{-u}\le e^{-1}$. Therefore,
\begin{equation*}
 \langle (e^{-u}\,\z,Du),\varphi\rangle =-\int_\Omega \varphi \, u \, \Div(e^{-u}\z)- \int_\Omega u\,e^{-u}\z \cdot \nabla \varphi \, dx
\end{equation*}
is a well--defined distribution (although the distribution $(\z, Du)$ is not). Moreover, we may apply the Anzellotti procedure and obtain a Radon measure.

Taking this fact in mind, we may follow the proof of \cite[Theorem 4.3]{LS}. Starting from suitable approximating problems, we get a limit of the approximate solutions $u\in BV(\Omega)$ such that $D^ju=0$. In addition, we also get a vector field $\z\in \DM (\Omega)$ such that $\|\z\|_\infty\le1$. Moreover, the equation
\eqref{cond-ditribucion} holds and
\begin{equation*}
  -\Div(e^{-u}\z)=e^{-u}f\,.
\end{equation*}
This last equality implies that $u$ is integrable with respect to the measure $\Div(e^{-u}\z)$ and so $(e^{-u}\,\z,Du)$ is a Radon measure.

Two conditions of Definition \ref{def-sol} must still be proved, namely, \eqref{cond-medida} and \eqref{cond-frontera}. We begin by seeing
\begin{equation}\label{medidas-T_k}
(\z, DT_k(u))=|DT_k(u)| \,\text{ as measures in }\, \Omega
\end{equation}
for every $k >0$.

To see \eqref{medidas-T_k} we start with the following equality as measures (proved in \cite[Theorem 4.3]{LS}):
\begin{equation}\label{des-previa}
|De^{-u}| \le (e^{-u}\z,Du)\,.
\end{equation}
First, we will show
\begin{equation*}
|De^{-T_k(u)}| \le (e^{-u}\z,DT_k(u)) \,.
\end{equation*}
On the one hand, considering the restriction to the set $\{ u \ge k\}$ we have
\begin{equation*}
|De^{-T_k(u)}|\res{\{ u \ge k\}} = e^{-T_k(u)}|DT_k(u)|\res{\{ u \ge k\}} = 0 \,,
\end{equation*}
and on the other hand
\begin{equation*}
 | (e^{-u}\z,DT_k(u))|\res{\{ u \ge k\}}\le  |DT_k(u))|\res{\{ u \ge k\}}= 0 \,.
\end{equation*}
Now, we just work with the restriction to the set $\{ u < k\}$. For every $\varphi\in C_0^\infty(\Omega)$ such that $\varphi\ge0$, using the definition of the distribution and applying \eqref{des-previa} we arrive at
\begin{multline*}
\langle (e^{-u}\z,DT_k(u))\res{\{ u < k\}}, \varphi \rangle =-\int_{\{ u < k\}} \varphi\,u\,\Div(e^{-u}\z) - \int_{\{ u < k\}}u\,e^{-u} \z\cdot\nabla \varphi \,dx
\\=\langle (e^{-u}\z,Du)\res{\{ u < k\}}, \varphi \rangle \ge \int_{\{ u < k\}} \varphi \,|De^{-u}|
\\ = \int_{\Omega} \varphi \,e^{-u} |DT_k(u)| = \int_{\Omega}\varphi \, |De^{-T_k(u)}| \,.
\end{multline*}

Now, we have to prove that $(\z,DT_k(u)) = |DT_k(u)|$ as measures in $\Omega$.
We use Proposition \ref{Prop2.3} and the Chain Rule to get
\begin{equation*}
|De^{-T_k(u)}|\le (e^{-u}\z , DT_k(u)) =e^{-u}(\z , DT_k(u)) \le e^{-u}|DT_k(u)| = |De^{-T_k(u)}| \,.
\end{equation*}
Then, the inequality becomes equality and $e^{-u}(\z,DT_k(u)) = e^{-u}|DT_k(u)|$ as measures in $\Omega$. We deduce that
\begin{equation*}
(\z,DT_k(u)) = |DT_k(u)|
\end{equation*}
as measures in $\Omega$, since $e^{-u} =0$ yields $T_k(u)=k$ for every $k >0$.

To check the boundary condition \eqref{cond-frontera} we consider the real function defined by
\begin{equation*}
J_1(s)=\int_0^sT_1(\sigma)\, d\sigma\,.
\end{equation*}
Then, in the same way than in \cite[Theorem 4.3]{LS}, we obtain
\begin{multline}\label{des-front}
\int_\Omega|DT_1(u)|+\int_{\partial\Omega}|T_1(u)|\, d\h^{N-1}+\int_\Omega|DJ_1(u)|+\int_{\partial\Omega}|J_1(u)|\, d\h^{N-1}\\
\le \int_\Omega f\,T_1(u)\, dx\,.
\end{multline}
Using the equation and the previous step, and applying the Green's formula and the Chain Rule, we get
\begin{multline*}
\int_\Omega f\,T_1(u)\, dx=-\int_\Omega T_1(u)\,\Div\z+\int_\Omega T_1(u)|Du|
\\=\int_\Omega(\z,DT_1(u))-\int_{\partial\Omega}T_1(u)[\z,\nu]\, d\h^{N-1}+\int_\Omega |DJ_1(u)|
\\=\int_\Omega|DT_1(u)|-\int_{\partial\Omega}T_1(u)[\z,\nu]\, d\h^{N-1}+\int_\Omega |DJ_1(u)|\,.
\end{multline*}
Going back to \eqref{des-front} and simplifying, it follows that
\begin{equation*}
  \int_{\partial\Omega}|T_1(u)|+T_1(u)[\z,\nu]\, d\h^{N-1}+\int_{\partial\Omega}|J_1(u)|\, d\h^{N-1}\le0\,.
\end{equation*}
Observe that both integrals are non--negative, so that both vanish. In particular, $J_1(u)=0$ $\h^{N-1}$--a.e. on $\partial\Omega$.
Therefore, the boundary condition holds true.
\end{proof} 

\subsection{Comparison principle}

Before proving the comparison principle we need to present some preliminary results.
\begin{Proposition}\label{prop-medidas}
Let $\z$ be a vector field in $\DM(\Omega)$ and let $u$ be a function of bounded variation with $D^ju=0$ and such that $(\z,DT_k(u))=|DT_k(u)|$ for every $k>0$. If $g:\Omega \to \R$ is a bounded, increasing and Lipschitz function, then  $(\z,Dg(u))=|Dg(u)|$ holds as measures.
\end{Proposition}
\begin{proof}
Since $(\z,DT_k(u))=|DT_k(u)|$, the Radon--Nikod\'ym derivative of $(\z,DT_k(u))$ with respect its total variation $|DT_k(u)|$ is $\theta(\z,DT_k(u),x) =1$. Moreover, using Proposition \ref{Prop2.2} we get
\begin{equation*}
\theta(\z,Dg(T_k(u)),x)=\theta(\z,DT_k(u),x)=1 \,,
\end{equation*}
that is, $(\z,Dg(T_k(u)))=|Dg(T_k(u))|$ for every $k>0$. Now, by the Dominated Convergence Theorem, we take limits in this expression when $k$ goes to $\infty$ and it leads to
\begin{equation*}
(\z, Dg(u)) = g'(u)\,|Du| = |Dg(u)| \,.
\end{equation*}
\end{proof}

\begin{Proposition}\label{prop-1}
Let $f \in L^1(\Omega)$. If $u \in BV(\Omega)$ is a solution to problem \eqref{prob-prin} and $\z \in \DM(\Omega)$ is the associated vector field, then the following equality holds:
\begin{equation*}
-\Div(e^{-u}\z)=e^{-u}f  \,\text{ in }\,\dis (\Omega)\,.
\end{equation*}
\end{Proposition}
\begin{proof}
Let $\varphi \in C_0^\infty(\Omega)$, we take the test function $e^{-u}\varphi$ in problem \eqref{prob-prin} and we obtain
\begin{equation*}
-\int_\Omega e^{-u}\varphi\, \Div \z + \int_\Omega e^{-u}\varphi \,|Du| = \int_\Omega e^{-u}\varphi f \,dx \,.
\end{equation*}
Now, since $e^{-u}$ is bounded, we can use the definition of pairing $(\z,De^{-u})$ and the former equality becomes
\begin{equation*}
\int_\Omega e^{-u} \z \cdot \nabla \varphi \,dx +\int_\Omega \varphi\,(\z , De^{-u})+\int_\Omega e^{-u}\varphi \,|Du| =\int_\Omega e^{-u}\varphi f \,dx \,.
\end{equation*}
Finally, using $(\z,De^{-u})=-e^{-u}|Du|$ (see Proposition \ref{prop-medidas}) and Green's formula we deduce
\begin{equation*}
-\Div(e^{-u}\z)=e^{-u}f \,\text{ in }\,\dis (\Omega)\,.
\end{equation*}
\end{proof}

\begin{Theorem}
Let $f_1$ and $f_2$ be two non--negative functions in $ L^1(\Omega)$ with $f_1\le f_2$, and consider problems
\begin{equation}\label{prob-u1}
\left\{\begin{array}{ll}
\displaystyle -\Div\left(\frac{Du_1}{|Du_1|}\right)+|Du_1|=f_1(x)&\hbox{ in }\,\Omega\,,\\[3mm]
u_1=0 &\hbox{ on }\,\partial\Omega\,,
\end{array}\right.
\end{equation}
and
\begin{equation}\label{prob-u2}
\left\{\begin{array}{ll}
\displaystyle -\Div\left(\frac{Du_2}{|Du_2|}\right)+|Du_2|=f_2(x)&\hbox{ in }\,\Omega\,,\\[3mm]
u_2=0 &\hbox{ on }\,\partial\Omega\,.
\end{array}\right.
\end{equation}
If $u_1$ is a solution to problem \eqref{prob-u1} and $u_2$ is a  solution to problem \eqref{prob-u2}, then $u_1\le u_2$.
\end{Theorem}

\begin{proof}
For each $i=1,2$, we know that solution $u_i\in BV(\Omega)$ satisfies $D^ju_i=0$ and there exists a vector field $\z_i \in \DM(\Omega)$ such that $\|\z_i\|_\infty\le 1$. Moreover,
\begin{equation*}
-\Div \z_i +|Du_i| =f_i \,\text{ in }\,\dis (\Omega)\,,
\end{equation*}
\begin{equation*}
(\z_i, DT_k(u_i))=|DT_k(u_i)| \,\text{ as measures in }\, \Omega\;\; (\mbox{for every } k >0)\,,
\end{equation*}
and
\begin{equation*}
 u_i\big|_{\partial \Omega} = 0\,.
\end{equation*}

\medskip
We are seeking that $u_1\le u_2$, to this end we divide the proof in several steps.

\bigskip
\textbf{STEP 1: ${\bf (\z_1-\z_2,D(T_k(u_1)-T_k(u_2))^+)}$ is a positive Radon measure for all ${\bf k>0}$.}

Let $\varphi \in C_0^\infty(\Omega)$ with $\varphi \ge 0$. Then, the measure $(\z_1-\z_2,D(T_k(u_1)-T_k(u_2))^+))$ actually is
\begin{multline*}
\int_\Omega \varphi \,(\z_1-\z_2,D(T_k(u_1)-T_k(u_2))^+) = \int_{\{T_k(u_1)> T_k(u_2)\}} \varphi \,(\z_1-\z_2,D(T_k(u_1)-T_k(u_2)))
\\= \int_{\{T_k(u_1)> T_k(u_2)\}} \varphi\bigg[(z_1,DT_k(u_1)) - (z_2,DT_k(u_1)) - (z_1,DT_k(u_2)) + (z_2,DT_k(u_2)) \bigg]
\\=\int_{\{T_k(u_1)> T_k(u_2)\}} \varphi \bigg[ |DT_k(u_1)| - (z_2,DT_k(u_1)) -(z_1,DT_k(u_2)) + |DT_k(u_2)| \bigg]
\\ \ge 0 \,,
\end{multline*}
because $(z_i,Du_j)\le |Du_j|$ for $i,j=1,2$.

Therefore, we conclude that $(\z_1-\z_2,D(T_k(u_1)-T_k(u_2))^+)$ is a positive Radon measure.

\bigskip
\textbf{STEP 2: Prove that ${\bf \displaystyle \int_{\{u_1>u_2\}} (e^{-u_2}- e^{-u_1})(|Du_1| -|Du_2|)\ge 0}$.}

First, we take the test function $(e^{-u_2}-e^{-u_1})^+$ in problem \eqref{prob-u1} and since $(e^{-u_2}-e^{-u_1})^+\not=0$ if $u_1> u_2$, we get
\begin{multline}\label{ec-u1}
\int_{\{u_1>u_2\}} (\z_1,D(e^{-u_2}-e^{-u_1})) + \int_{\{u_1>u_2\}} (e^{-u_2}-e^{-u_1}) \,|Du_1|
\\ = \int_\Omega (e^{-u_2}-e^{-u_1})^+ f_1\,dx\,.
\end{multline}
Moreover, using that $e^{-u_2}-e^{-u_1} = (1-e^{-u_1})-(1-e^{-u_2})$ we also have
\begin{multline}\label{ec2-u1}
\int_\Omega (e^{-u_2}-e^{-u_1})^+ f_1\,dx =\int_{\{u_1>u_2\}}(\z_1,D(1-e^{-u_1}))- \int_{\{u_1>u_2\}}(\z_1,D(1-e^{-u_2}))
\\[2mm]+ \int_{\{u_1>u_2\}} e^{-u_2}\,|Du_1|- \int_{\{u_1>u_2\}} e^{-u_1} \,|Du_1|
\\[2mm]=\int_{\{u_1>u_2\}} |D(1-e^{-u_1})| - \int_{\{u_1>u_2\}}(\z_1,D(1-e^{-u_2}))
\\[2mm]+ \int_{\{u_1>u_2\}} e^{-u_2}\,|Du_1|- \int_{\{u_1>u_2\}} e^{-u_1} \,|Du_1|
\\[2mm] =-\int_{\{u_1>u_2\}}(\z_1,D(1-e^{-u_2})) + \int_{\{u_1>u_2\}} e^{-u_2}\,|Du_1|\,,
\end{multline}
where we have used Proposition \ref{Prop2.2} and the Chain Rule.

Now, taking the same test function $(e^{-u_2}-e^{-u_1})^+$ in problem \eqref{prob-u2} and making similar computations we obtain
\begin{equation}\label{ec2-u2}
\int_\Omega (e^{-u_2}-e^{-u_1})^+ f_2\,dx =\int_{\{u_1>u_2\}}(\z_2,D(1-e^{-u_1})) - \int_{\{u_1>u_2\}} e^{-u_1}\,|Du_2|\,.
\end{equation}
Since $f_1\le f_2$, we can join expressions \eqref{ec2-u1} and \eqref{ec2-u2} to get the following inequality:
\begin{multline*}
\int_{\{u_1>u_2\}} e^{-u_1}\,|Du_2|+ \int_{\{u_1>u_2\}} e^{-u_2}\,|Du_1|
\\ \le \int_{\{u_1>u_2\}}(\z_1,D(1-e^{-u_2}))  +\int_{\{u_1>u_2\}}(\z_2,D(1-e^{-u_1}))
\\ \le \int_{\{u_1>u_2\}} |(\z_1,D(1-e^{-u_2}))| +\int_{\{u_1>u_2\}} |(\z_2,D(1-e^{-u_1}))|
\\ \le \int_{\{u_1>u_2\}} |D(1-e^{-u_2})| + \int_{\{u_1>u_2\}} |D(1-e^{-u_1})|
\\= \int_{\{u_1>u_2\}} e^{-u_2}|Du_2| + \int_{\{u_1>u_2\}} e^{-u_1}|Du_1|\,,
\end{multline*}
where we have used that $\|\z_i\|\le 1$ for $i=1,2$ and the Chain Rule.

In conclusion, we have just proved
\begin{equation*}
 \int_{\{u_1>u_2\}} e^{-u_2}|Du_2| + \int_{\{u_1>u_2\}} e^{-u_1}|Du_1| -\int_{\{u_1>u_2\}} e^{-u_1}\,|Du_2|-\int_{\{u_1>u_2\}} e^{-u_2}\,|Du_1| \le 0\,,
\end{equation*}
and we are done.

\bigskip
\textbf{STEP 3: The Radon measure ${\bf (\z_1-\z_2,D(T_k(u_1)-T_k(u_2))^+)}$ vanishes for all ${\bf k>0}$.}

Since $u_1$ is a solution to problem \eqref{prob-u1} and $u_2$ is a solution to problem \eqref{prob-u2}, the following equalities hold in $\dis(\Omega)$ (see Proposition \ref{prop-1}):
\begin{equation}\label{dis-u1}
-\Div (e^{-u_1}\z_1)=e^{-u_1}f_1
\end{equation}
and
\begin{equation}\label{dis-u2}
-\Div (e^{-u_2}\z_2)=e^{-u_2}f_2\,.
\end{equation}

Firstly, let $k>0$ we choose the test function $(T_k(u_1)-T_k(u_2))^+$ in equality \eqref{dis-u1}. Applying Green's formula  we get
\begin{equation}\label{dis2-u1}
\int_\Omega (e^{-u_1}\z_1,D(T_k(u_1)-T_k(u_2))^+) = \int_\Omega (T_k(u_1)-T_k(u_2))^+ e^{-u_1}f_1\,dx \,,
\end{equation}
and using the same test function but now in equality \eqref{dis-u2} we have
\begin{equation}\label{dis2-u2}
\int_\Omega (e^{-u_2}\z_2,D(T_k(u_1)-T_k(u_2))^+) = \int_\Omega (T_k(u_1)-T_k(u_2))^+ e^{-u_2}f_2\,dx \,.
\end{equation}
Now, we put together \eqref{dis2-u1} and \eqref{dis2-u2} to obtain
\begin{multline}\label{ec-2}
\int_\Omega (T_k(u_1)-T_k(u_2))^+ (e^{-u_1}f_1-e^{-u_2}f_2)\,dx\\
=\int_\Omega (e^{-u_1}\z_1 - e^{-u_2}\z_2,D(T_k(u_1)-T_k(u_2))^+)
\\=\int_{\{ T_k(u_1)>T_k(u_2) \}} (e^{-u_2}\z_2 - e^{-u_1}\z_1,D(T_k(u_2)-T_k(u_1)))\,.
\end{multline}
Observe that the integral on the left hand side is non--positive since $e^{-u_1}f_1-e^{-u_2}f_2\le0$ where
$T_k(u_1)-T_k(u_2)>0$.
Our aim is to prove the following limit:
\begin{equation}\label{limite}
\lim_{k \to \infty}  \int_\Omega (e^{-u_1}\z_1 - e^{-u_2}\z_2,D(T_k(u_1)-T_k(u_2))^+) =0 \,,
\end{equation}
which is non--positive because of \eqref{ec-2}. To this end, we write
 \begin{multline*}
\int_{\{ T_k(u_1)>T_k(u_2) \}} (e^{-u_2}\z_2 - e^{-u_1}\z_1,D(T_k(u_2)-T_k(u_1)))
\\= \underbrace{\int_{\{ T_k(u_1)>T_k(u_2) \}} ((e^{-u_2}-e^{-u_1})\,\z_2,D(T_k(u_2)-T_k(u_1)))}_{\mbox{(I.1)}}
\\ + \underbrace{\int_{\{ T_k(u_1)>T_k(u_2) \}} (e^{-u_1}(\z_2-\z_1),D(T_k(u_2)-T_k(u_1)))}_{\mbox{(I.2)}} \,,
\end{multline*}
and will see that the limits as $k$ goes to $\infty$ of (I.1) and of (I.2) are non--negative and so \eqref{limite} holds.

On the one hand, we know that
\begin{multline*}
\int_{\{ T_k(u_1)>T_k(u_2) \}} ((e^{-u_2}-e^{-u_1})\,\z_2,D(T_k(u_2)-T_k(u_1)))
\\ \ge  \int_{\{ T_k(u_1)>T_k(u_2) \}} (e^{-u_2}-e^{-u_1})\car_{\{u_2<k\}}|Du_2| - \int_{\{ T_k(u_1)>T_k(u_2) \}} (e^{-u_2}-e^{-u_1})\car_{\{u_1<k\}}|Du_1| \,,
\end{multline*}
and when we take limits when $k$ goes to $\infty$, we get
\begin{multline*}
\lim_{k \to \infty} \int_{\{ T_k(u_1)>T_k(u_2) \}} ((e^{-u_2}-e^{-u_1})\,\z_2,D(T_k(u_2)-T_k(u_1)))
\\=\int_{\{u_1>u_2\}}  (e^{-u_2}-e^{-u_1}) (|Du_2|-|Du_1|) \ge 0 \,,
\end{multline*}
which is non--negative due to Step 2.

On the other hand, we already know that integral (I.2) is non--negative (because of Step 1), therefore the limit when $k \to \infty$ is non--negative too.

In short, we have proved
\begin{multline*}
\lim_{k \to \infty}\int_{\{ T_k(u_1)>T_k(u_2) \}} (T_k(u_1)-T_k(u_2)) (e^{-u_1}f_1-e^{-u_2}f_2)\,dx
\\ = \lim_{k \to \infty}\int_\Omega (e^{-u_1}\z_1 - e^{-u_2}\z_2,D(T_k(u_1)-T_k(u_2))^+) =0 \,.
\end{multline*}
Furthermore, since \eqref{ec-2}=(I.1)+(I.2) and the limits of integral (I.1) and (I.2) are both non--negative, it follows that both limits vanish.

Now, some remarks on Radon--Nikod\'ym derivatives of these measures are in order.
Let $\theta_k^1(\z_2,DT_k(u_1),x)$ be the Radon--Nikod\'ym derivative of $(\z_2,DT_k(u_1))$ with respect to $|DT_k(u_1)|$:
\begin{equation*}
\theta_k^1(\z_2,DT_k(u_1),x)\,|DT_k(u_1)|=(\z_2,DT_k(u_1)) \,.
\end{equation*}
Since $|(\z_2,DT_k(u_1))| \le |DT_k(u_1)|$, it follows that $|\theta_k^1(\z_2,DT_k(u_1),x)| \le 1$.
We point out that this function is $|DT_k(u_1)|$--measurable and, taking $\theta^1_k(\z_2,DT_k(u_1),x)=0$ in $\{ u_1\ge k\}$, it is $|Du_1|$--measurable.

On the other hand, it holds that $(\z_2,DT_{k+1}(u_1))\res{\{ u_1< k \}}=(\z_2,DT_k(u_1))$. Therefore
\begin{equation*}
\theta^1_{k+1}(\z_2,DT_{k+1}(u_1),x)\car_{\{ u_1<k \}}(x)=\theta^1_k(\z_2,DT_k(u_1),x) \,,
\end{equation*}
and $\theta_k^1(\z_2,DT_k(u_1),x)$ defines a non--decreasing sequence of $|Du_1|$--measurable functions.

Likewise, if we denote by $\theta_k^2(\z_1,DT_k(u_2),x)$ the Radon--Nikod\'ym derivative of $(\z_1,DT_k(u_2))$ with respect to $|DT_k(u_2)|$, then we may deduce the inequality $|\theta_k^2(\z_1,DT_k(u_2),x)| \le 1$. Moreover, $\theta_k^2(\z_1,DT_k(u_2),x)$ defines a non--decreasing sequence of $|Du_2|$--measurable functions.

Now, we define the functions $\theta^1(x)$ and $\theta^2(x)$ such that
\begin{equation*}
\theta^1(x)=\theta_k^1(\z_2,DT_k(u_1),x) \quad\mbox{ if }\quad u_1(x)<k \,,
\end{equation*}
and
\begin{equation*}
\theta^2(x)=\theta_k^2(\z_1,DT_k(u_2),x) \quad \mbox{ if }\quad u_2(x)<k \,.
\end{equation*}
We know that $\theta^1$ and $\theta^2$ are $|Du_1|$ and $|Du_2|$--measurable, respectively, and satisfy $|\theta^1| \le 1$ and $|\theta^2| \le 1$.

So let us get back to expression (I.2). We know that
\begin{multline*}
\int_\Omega (e^{-u_1}(\z_1-\z_2),D(T_k(u_1)-T_k(u_2))^+)
\\=\int_{\{ T_k(u_1)>T_k(u_2) \}} e^{-u_1}\bigg[ (\z_1,DT_k(u_1))-(\z_2,DT_k(u_1))-(\z_1,DT_k(u_2))+(\z_2,DT_k(u_2)) \bigg]
\\=\int_\Omega  e^{-u_1}\car_{\{ T_k(u_1)>T_k(u_2) \}\cap \{ u_1<k\}} (1-\theta^1(x))|Du_1|
\\+ \int_\Omega  e^{-u_1}\car_{\{ T_k(u_1)>T_k(u_2)\}\cap \{ u_2<k\}} (1-\theta^2(x))|Du_2| \,,
\end{multline*}
and using the Convergence Dominated Theorem we can take limit when $k \to \infty$ to arrive at
\begin{equation*}
0=\int_{\{ u_1>u_2 \}} e^{-u_1} (1-\theta^1(x))|Du_1|
+\int_{\{ u_1>u_2 \}} e^{-u_1} (1-\theta^2(x))|Du_2|  \,.
\end{equation*}
Since both integrals are non--negative, it yields
\begin{equation*}
0 =\int_{\{ u_1>u_2 \}} e^{-u_1} (1-\theta^1(x))|Du_1|=\int_{\{ u_1>u_2 \}} e^{-u_1} (1-\theta^2(x))|Du_2| \,.
\end{equation*}
Therefore, we deduce that $1-\theta^i(x)=0$ $|Du_i|$--a.e. in $\{ u_1>u_2 \}$ for $i=1,2$ and then, the Radon--Nikod\'ym derivative is $\theta_k^i=1$ $|Du_i|$--a.e. in $\{ u_1>u_2 \} \cap \{u_i<k\}$ with $i=1,2$ for every $k>0$. That is, we have the following equalities as measures:
\begin{equation}\label{med-1}
|DT_k(u_1)|\res{\{ u_1>u_2 \}}=(\z_2,DT_k(u_1))\res{\{ u_1>u_2 \}} \,,
\end{equation}
and
\begin{equation}\label{med-2}
|DT_k(u_2)|\res{\{ u_1>u_2 \}}=(\z_1,DT_k(u_2)) \res{\{ u_1>u_2 \}} \,.
\end{equation}
Finally, noting that $\{T_k(u_1)>T_k(u_2)\}\subseteq \{u_1>u_2\}$ and the measure $(\z_1-\z_2,D(T_k(u_1)-T_k(u_2))^+) $ is non--negative:
\begin{multline*}
(\z_1-\z_2,D(T_k(u_1)-T_k(u_2))^+)
\\=\bigg[|DT_k(u_1)|-(\z_2,DT_k(u_1))-(\z_1,DT_k(u_2))+|DT_k(u_2)|\bigg]\res{\{T_k(u_1)>T_k(u_2)\}}
\\ \le \bigg[|DT_k(u_1)|-(\z_2,DT_k(u_1))-(\z_1,DT_k(u_2))+|DT_k(u_2)|\bigg]\res{\{u_1>u_2\}}
\\=0 \,.
\end{multline*}

\bigskip
\textbf{STEP 4: ${\bf (\z_i,DT_k(u_j))\res{\{T_k(u_1)>T_k(u_2)\}}=|DT_k(u_j)|\res{\{T_k(u_1)>T_k(u_2)\}}}$ as measures for ${\bf i,j=1,2}$ and ${\bf k>0}$.}

Since $\{T_k(u_1)>T_k(u_2)\}\subseteq \{u_1>u_2\}$ and we have proved equalities \eqref{med-1} and \eqref{med-2}, Step 4 is straightforward.

\bigskip
\textbf{STEP 5: If ${\bf u_1>u_2}$, then ${\bf f_1=f_2 =0}$.}

In Step 3 we have proved that the limit of expression \eqref{ec-2} when $k$ goes to $\infty$ is 0. Then, using the Monotone Convergence Theorem, we get
\begin{equation*}
0=\int_\Omega (u_1-u_2)^+(e^{-u_1}f_1-e^{-u_2}f_2) \,dx \,.
\end{equation*}
Notice that if $u_1>u_2$, then $e^{-u_1}f_1= e^{-u_2}f_2$ and $f_1 =e^{-(u_2-u_1)} f_2 > f_2$ when $f_2\ne0$. We conclude that $u_1>u_2$ implies $f_2=f_1=0$.

\bigskip
\textbf{STEP 6: Prove that ${\bf \displaystyle \int_{\{ u_1>u_2 \}} |Du_1| = \int_{\{ u_1>u_2 \}} |Du_2|}$.}

Firstly, we take $T_{\e}((T_k(u_1)-T_k(u_2))^+)$ as a test function in problems \eqref{prob-u1} and \eqref{prob-u2} and using the previous step we get the following equalities:
\begin{multline}\label{ec3-u1}
0 = \int_{\{ T_k(u_1)>T_k(u_2) \}} (\z_1,DT_{\e}(T_k(u_1)-T_k(u_2))) \\+ \int_{\{ T_k(u_1)>T_k(u_2) \}} T_{\e}(T_k(u_1)-T_k(u_2)) \,|Du_1| \,,
\end{multline}
and
\begin{multline}\label{ec3-u2}
0 = \int_{\{ T_k(u_1)>T_k(u_2) \}} (\z_2,DT_{\e}(T_k(u_1)-T_k(u_2))) \\+ \int_{\{ T_k(u_1)>T_k(u_2) \}} T_{\e}(T_k(u_1)-T_k(u_2)) \,|Du_2| \,.
\end{multline}
Now, we use Step 3 to have:
\begin{equation*}
(\z_1-\z_2, D(T_k(u_1)-T_k(u_2)))\res{\{ T_k(u_1)>T_k(u_2) \}} = 0 \,.
\end{equation*}
Furthermore, when we take the restriction to the set $\{ 0 <T_k(u_1)-T_k(u_2)<\e \}$ for all $\e >0$, we also get that the measure vanishes.
Due to this fact, when we consider together equations \eqref{ec3-u1} and \eqref{ec3-u2}, we obtain
\begin{multline*}
\int_{\{ T_k(u_1)>T_k(u_2) \}} T_{\e}(T_k(u_1)-T_k(u_2)) \,|Du_1|\\=\int_{\{ T_k(u_1)>T_k(u_2) \}} T_{\e}(T_k(u_1)-T_k(u_2)) \,|Du_2| \,.
\end{multline*}
Now, dividing both integrals by $\e$ and using the Dominated Convergence Theorem we can take limits as $\e$ goes to 0 and then we get
\begin{equation*}
\int_{\{ T_k(u_1)>T_k(u_2) \}} |Du_1|=\int_{\{ T_k(u_1)>T_k(u_2) \}} |Du_2| \,.
\end{equation*}
Finally, the Dominated Convergence Theorem also allows us to take limits as $k \to \infty$ and so
we arrive at
\begin{equation}\label{ec-3}
\int_{\{ u_1>u_2 \}} |Du_1|=\int_{\{ u_1>u_2 \}} |Du_2| \,.
\end{equation}

\bigskip
\textbf{STEP 7: Prove ${\bf Du_1 =Du_2 =0}$ in ${\bf \{ u_1 > u_2 \}}$.}

We begin taking the test function $(T_k(u_1)-T_k(u_2))^+$ in problem \eqref{prob-u1} and having in mind Step 4 and Step 5, we get:
\begin{multline}\label{ec-4}
0= \int_{\{ T_k(u_1) > T_k(u_2) \}} |DT_k(u_1)|-\int_{\{ T_k(u_1) > T_k(u_2) \}} |DT_k(u_2)|
\\+ \int_{\{ T_k(u_1) > T_k(u_2) \}} (T_k(u_1)-T_k(u_2))\,|Du_1| \,.
\end{multline}

Now, the first two integrals are convergent as $k \to \infty$ due to the Dominated Convergence Theorem and the last one converges by the Monotone Convergence Theorem. Hence, when $k$ goes to $\infty$ in \eqref{ec-4} we get
\begin{equation*}
0= \int_{\{ u_1>u_2 \}} (|Du_1|-|Du_2|) + \int_{\{ u_1>u_2 \}} (u_1-u_2)\,|Du_1| \,,
\end{equation*}
and since the first integral is finite, the last one is finite too.

On the other hand, we have proved in Step 6 that the first integral vanishes, then the above equality becomes
\begin{equation*}
0  = \int_{\{ u_1 > u_2 \}} (u_1-u_2)\,|Du_1| \,,
\end{equation*}
and we deduce that $|Du_1|\res{\{ u_1 > u_2 \}} =0$ and also $Du_1 =0$ in $\{ u_1 > u_2 \}$.

To prove that $Du_2 =0$ in $\{ u_1 > u_2 \}$ we use \eqref{ec-3} and since we already know that $Du_1 =0$ in $\{ u_1 > u_2 \}$, it becomes
\begin{equation*}
0=\int_{\{ u_1 > u_2 \}} |Du_2| \,.
\end{equation*}
Therefore we have that $Du_2 =0$ in $\{ u_1 > u_2 \}$.

\bigskip
\textbf{STEP 8:  ${\bf u_1\le u_2}$ in ${\bf \Omega}$.}

We have seen that $D(u_1-u_2) =0$ in $\{ u_1 > u_2 \}$ and since $D^j (u_1-u_2) =0$, it holds that $D(u_1-u_2)^+ =0$ in $\Omega$. Moreover, we know that $(u_1-u_2)^+=0$ in $\partial \Omega$, therefore we get that $0=(u_1-u_2)^+$ in $\Omega$.
\end{proof}

\begin{Corollary}
Let $\Omega$ be a bounded open subset of $\R^N$  with Lipschitz boundary. Let $f$ be a non--negative function in $L^1(\Omega)$. Then, problem
\begin{equation*}
\left\{\begin{array}{ll}
\displaystyle -\Div\left(\frac{Du}{|Du|}\right)+|Du|=f(x)&\hbox{ in }\,\Omega\,,\\[3mm]
u=0 &\hbox{ on }\,\partial\Omega\,,
\end{array}\right.
\end{equation*}
has a unique solution $u \in BV(\Omega)$.
\end{Corollary}

\section{Better summability}

In Section 3 we have seen that problem \eqref{prob-prin} has a solution for every non--negative datum of $L^1(\Omega)$, and this solution belongs to $BV(\Omega)\subset L^{\frac N{N-1}}(\Omega)$. We expect that the solution satisfies a better summability if the datum belongs to $L^p(\Omega)$, $p>1$. In this regard we recall that, when data $f$ are in the space $L^p(\Omega)$ with $p>N$, it is proved in \cite{MS} that the solution is always bounded. For datum $f \in L^N(\Omega)$, we proved in \cite{LS} that the solution belongs to $L^q(\Omega)$ with $1< q < \infty$.

In this section, we are showing that solutions belong to $L^{\frac{Np}{N-p}}(\Omega)$ if data are in $L^p(\Omega)$ with $1< p<N$. Observe that this result adjust continuously for $p=1$ and $p=N$ with the known facts.

\bigskip

\begin{center}
\begin{tabular}{|c|c|}
\hline
&\\
Data & Solution\\
&\\
\hline
\hline
&\\
$ f\in L^p(\Omega) $ with $p>N$ & $u\in L^\infty(\Omega)$\\
&\\
\hline
&\\
$f\in L^{N}(\Omega)$ & $u\in L^q(\Omega)$ with $1\le q < \infty$\\
&\\
\hline
&\\
$f\in L^p(\Omega)$ with $1<p<N$ & $u\in L^{\frac{Np}{N-p}}(\Omega)$\\
&\\
\hline
&\\
$f\in L^1(\Omega)$  & $u\in L^{\frac{N}{N-1}}(\Omega)$ \\
&\\
\hline
\end{tabular}
\end{center}

\bigskip

The proof of our theorem relies on certain preliminary results. The first one enable us to take a power of our solution $u^q$ as a test function in problem \eqref{prob-prin}.

\begin{Proposition}\label{proposition-1}
If $u\in BV(\Omega)$ is a solution to problem \eqref{prob-prin} satisfying $u^q \in L^{p'}(\Omega)$ for certain $q>1$, then $u^q$ and $u^{q+1} \in BV(\Omega)$. Moreover,
\begin{equation}\label{ec-prop}
\int_\Omega |Du^q|+\int_\Omega u^q\,|Du|=\int_\Omega u^q\,f \,.
\end{equation}
\end{Proposition}

\begin{proof}

Fixed $k>0$, we define the function $G_k(s)= s-T_k(s)$, and we take the test function $G_\delta(T_k(u)^q)$ with $\delta, k>0$ in problem \eqref{prob-prin} obtaining the following equality:
\begin{equation*}
\int_\Omega (\z,DG_\delta(T_k(u)^q)) + \int_\Omega G_\delta(T_k(u)^q)\,|Du| = \int_\Omega G_\delta(T_k(u)^q)f\,dx \,.
\end{equation*}
Since we know that the Radon--Nikod\'ym derivative of $(\z,DG_\delta(T_k(u)^q))$ and $(\z,DT_k(u))$ with respect their respective total variations are the same (Proposition \ref{Prop2.2})
and $(\z,DT_k(u))=|DT_k(u)|$ holds for all $k>0$, we deduce that
\begin{equation*}
(\z,G_\delta(T_k(u)^q))=|DG_\delta((T_k(u))^q)|\,.
\end{equation*}
Then we can write:
\begin{equation}\label{lim-ademas}
\int_\Omega |DG_\delta(T_k(u)^q)| + \int_\Omega G_\delta(T_k(u)^q)\,|Du| = \int_\Omega G_\delta(T_k(u)^q)f\,dx \,.
\end{equation}
Now, we use the inequality $G_\delta(T_k(u)^q) \le u^q$ and H\"older's inequality to get the following bound:
\begin{equation*}
\int_\Omega G_\delta(T_k(u)^q)f\,dx \le \int_\Omega u^q f\,dx \le \|u^q\|_{p'}\|f\|_p < \infty \,.
\end{equation*}
Therefore, each integral in left hand side of \eqref{lim-ademas} is also bounded:
\begin{equation}\label{BV-q}
\int_\Omega |DG_\delta(T_k(u)^q)| \le \|u^q\|_{p'}\|f\|_p < \infty \,,
\end{equation}
and
\begin{equation}\label{BV-q+1}
\int_\Omega G_\delta(T_k(u)^q)\,|Du| \le \|u^q\|_{p'}\|f\|_p < \infty \,.
\end{equation}
We will take advantage of these bounds to take limits in \eqref{lim-ademas}.

Now, we are able to prove $u^q \in BV(\Omega)$.
Using the Chain Rule in \eqref{BV-q} we can write the following inequalities
\begin{equation*}
\int_\Omega \car_{\{u<k\}\cap\{u^q > \delta\}}|Du^q| =\int_\Omega |DG_\delta(T_k(u)^q)| \le \|u^q\|_{p'}\|f\|_p < \infty \,,
\end{equation*}
and, using Monotone Convergence Theorem, we let $\delta \to 0^+$ to get
\begin{equation*}
\int_\Omega \car_{\{u<k\}}|Du^q| \le \|u^q\|_{p'}\|f\|_p < \infty \,.
\end{equation*}
Lastly, we let $k$ goes to $\infty$ and appealing to the Monotone Convergence Theorem once more, it works out that $u^q$ is a bounded variation function.

Let $0<\delta<1$ and keeping in mind \eqref{BV-q+1}, we get the following bound
\begin{multline*}
\int_\Omega \car_{\{u<k\}\cap\{u^{q+1} > \delta\}}|Du^{q+1}|
= (q+1)\int_\Omega u^q \car_{ \{u<k\} \cap\{u^{q+1}>\delta\}} |Du|
\\ \le (q+1)\int_\Omega (G_\delta(T_k(u)^q) + \delta)\,|Du|
\\ \le (q+1)(\|u^q\|_{p'}\|f\|_p +\delta\|u\|_{BV})< \infty \,.
\end{multline*}
Taking limits when $\delta \to 0^+$ and also when $k \to \infty$ we get
\begin{equation}\label{exp-1}
\int_\Omega |Du^{q+1}| \le (q+1)\|u^q\|_{p'}\|f\|_p < \infty \,,
\end{equation}
that is, $u^{q+1}\in BV(\Omega)$.

To conclude, we take limits in \eqref{lim-ademas} firstly when $\delta \to 0^+$ and secondly when $k \to \infty$ and then we obtain
\begin{equation*}
\int_\Omega |Du^q| + \int_\Omega u^q |Du| = \int_\Omega u^q \, f \,.
\end{equation*}
\end{proof}

\begin{Theorem}\label{teo-induccion}
Let $1<p<N$ and let $f \in L^p(\Omega)$ be a non--negative function. Then the solution to problem \eqref{prob-prin} belongs to $BV(\Omega)\cap L^{s}(\Omega)$ for every $1\le s<\frac{Np}{N-p}$.
\end{Theorem}

\begin{proof}
Let $u \in BV(\Omega)$ denote the unique solution to problem \eqref{prob-prin}. For every $j\in \N$, we will prove that $u \in L^{s_j}(\Omega)$ where
\begin{equation*}
s_j= N'\sum\limits_{k=0}^j \left(\dfrac{N'}{p'}\right)^k \,.
\end{equation*}
It should be noted that $\lim\limits_{j\to\infty}s_j=N'\sum\limits_{k=0}^\infty \left(\dfrac{N'}{p'}\right)^k = \dfrac{Np}{N-p}$. Thus, proving $u \in L^{s_j}(\Omega)$ for all $j\in\N$, we are done.

Firstly, we choose $q=\frac{N'}{p'}$ and since $u^q \in L^{p'}(\Omega)$, we may apply Proposition \ref{proposition-1} to conclude that $u^{q+1} \in BV(\Omega)\subseteq L^{N'}(\Omega)$ and therefore $u \in L^{N'\left(\frac{N'}{p'}+1\right)}(\Omega)$, that is, $u\in L^{s_1}(\Omega)$.

Assuming now that $u \in L^{s_j}(\Omega)$, we take
\begin{equation*}
q=\frac{N'}{p'}\sum_{k=0}^j \left(\frac{N'}{p'}\right)^k \,.
\end{equation*}
By hypothesis, we already know that $u \in L^{qp'}(\Omega)$, and using Proposition \ref{proposition-1} we get $u^{q+1} \in BV(\Omega) \subseteq L^{N'}(\Omega)$. Hence, $u \in L^{N'(q+1)}(\Omega)= L^{s_{j+1}}(\Omega)$.
\end{proof}

Now we are ready to prove the main result of this section.

\begin{Theorem}\label{teo-regularidad}
Let $f$ be a non--negative function belonging to $L^p(\Omega)$ with $1<p<N$. Then the unique solution $u$ to problem \eqref{prob-prin} satisfies $u \in BV(\Omega)\cap L^{\frac{Np}{N-p}}(\Omega)$.
\end{Theorem}

\begin{proof}
To show that $u \in L^{\frac{Np}{N-p}}(\Omega)$, we first claim that inequality \eqref{claim0} below holds for every $0<q<\dfrac{N(p-1)}{N-p}$.

If we choose $0<q<\dfrac{N(p-1)}{N-p}$, then we have that $qp'<\dfrac{N(p-1)}{N-p}\dfrac{p}{p-1}=\dfrac{Np}{N-p}$. Therefore, applying Theorem \ref{teo-induccion} and Proposition \ref{proposition-1} we arrive at $u^{q+1}\in BV(\Omega)$.

Now, we use Sobolev's inequality and inequality \eqref{exp-1} to get
\begin{multline}\label{exp-2}
\left( \int_\Omega u^{(q+1)N'} \,dx\right)^{\frac{1}{N'}} \le C(p,N) \int_\Omega |Du^{q+1}|
\\ \le C(p,N)(q+1)\|f\|_p \left(\int_\Omega u^{qp'} \,dx\right)^{\frac{1}{p'}} \,.
\end{multline}
Moreover, since $qp' < (q+1)N'$, we can apply H\"older's inequality and we also get
\begin{equation*}
\int_\Omega u^{qp'} \,dx \le \left(\int_\Omega (u^{qp'})^{\frac{(q+1)N'}{qp'}} \,dx\right)^{\frac{qp'}{(q+1)N'}} |\Omega|^{1-\frac{qp'}{(q+1)N'}} \,.
\end{equation*}
Summing up, we have
\begin{equation*}
\left( \int_\Omega u^{(q+1)N'} \,dx\right)^{\frac{1}{N'}}
 \le C(p,N)(q+1)\|f\|_p \left(\int_\Omega u^{(q+1)N'} \,dx \right)^{\frac{q}{(q+1)N'}} |\Omega|^{\frac{1}{p'}-\frac{q}{(q+1)N'}} \,,
\end{equation*}
that is, we have proved our claim:
\begin{equation}\label{claim0}
\left( \int_\Omega u^{(q+1)N'} \,dx \right)^{\frac{1}{N'}\left( 1-\frac{q}{q+1} \right)} \le C(p,N)(q+1)\|f\|_p |\Omega|^{\frac{1}{p'}-\frac{q}{(q+1)N'}} \,.
\end{equation}

Now, let $0<q_n<\dfrac{N(p-1)}{N-p}$ define a non--decreasing sequence convergent to $\frac{N(p-1)}{N-p}$. Hence, for every $n \in \N$ it holds
\begin{equation*}
\left( \int_\Omega u^{(q_n+1)N'} \,dx \right)^{\frac{1}{N'}\frac{1}{q_n +1}} \le C(p,N)(q_n+1)\|f\|_p |\Omega|^{\frac{1}{p'}-\frac{q_n}{(q_n+1)N'}} \,.
\end{equation*}
Thanks to Fatou's lemma, letting $n \to \infty$, we get
\begin{multline*}
\int_\Omega u^{\frac{p\,(N-1)}{N-p}N'} \,dx  \le \liminf_{n\to \infty}  \left[ C(p,N)(q_n+1)\|f\|_p |\Omega|^{\frac{1}{p'}-\frac{q_n}{(q_n+1)N'}} \right]^{N'(q_n+1)}
\\ \le \left[ C(p,N)\frac{p\,(N-1)}{N-p}\|f\|_p \right]^{\frac{Np}{N-p}} \,.
\end{multline*}

Therefore, $u \in L^{\frac{Np}{N-p}}(\Omega)$ holds.
\end{proof}

\begin{remark}
Going back to Proposition \ref{proposition-1}, it follows from $u^{\frac{N(p-1)}{N-p}} \in L^{p^\prime}(\Omega)$ that $u^{\frac{N(p-1)}{N-p}}$ can be taken as a test function in problem \eqref{prob-prin}, that is,
\[
\int_\Omega |D(u^{\frac{N(p-1)}{N-p}})|+\int_\Omega u^{\frac{N(p-1)}{N-p}}|Du|=\int_\Omega fu^{\frac{N(p-1)}{N-p}}\,.
\]
\end{remark}

\section{Explicit examples}

This section is devoted to show radial examples of solutions in a ball.
These examples allow us to provide evidence that our regularity result is sharp (see Remark \ref{regul} below).

In the sequel, we denote by $B_R(0)$ the open ball centered at $0$ and of radius $R$.

\begin{Example}\label{ej-1}
Let $R>0$, we consider problem
\begin{equation}\label{ejemplo-1}
\left\{\begin{array}{ll}
\displaystyle -\Div\left(\frac{Du}{|Du|}\right)+|Du|=\frac{\lambda}{|x|^q} & \hbox{ in
}B_R(0)\,,\\[3mm]
u=0 & \hbox{ on }\partial B_R(0)\,,
\end{array}\right.
\end{equation}
with $1<q<N$ and $\lambda >0$.
\end{Example}

We know that solution $u$ to problem \eqref{ejemplo-1} must be a non--negative function of bounded variation with no jump part and there also exists a vector field $\z \in \DM(\Omega)$ with $\|\z\|_\infty \le 1$  such that
\begin{equation}\label{cond-distribucion}
-\Div \z +|Du| =\frac{\lambda}{|x|^q} \,\text{ in }\,\dis (\Omega)\,,
\end{equation}
\begin{equation*}
(\z, DT_k(u))=|DT_k(u)| \,\text{ as measures in }\, \Omega \;\; (\mbox{for every } k >0)\,,
\end{equation*}
and
\begin{equation*}
u\big|_{\partial \Omega} = 0\,.
\end{equation*}

We assume that the solution is radial, that is, $u(x)=h(|x|)=h(r)$. Moreover, in order to satisfy the Dirichlet condition, we want that $h(R)=0$ holds. In addition, we also assume $h'(r)\le 0$ for all $0\le r\le R$.

If $h^\prime(r)<0$ in an interval, then the vector field is given by $\z(x)=\frac{h^\prime(|x|)}{|h^\prime(|x|)|}=-\frac{x}{|x|}$ and $\Div \z(x)= -\frac{N-1}{|x|}$.

Therefore, equation \eqref{cond-distribucion} becomes
\begin{equation}\label{dinar}
\frac{N-1}{r} - h'(r) = \frac{\lambda}{r^q} \,.
\end{equation}
Since we are assuming that $h'(r)<0$, then
\begin{equation*}
\frac{N-1}{r}- \frac{\lambda}{r^q} <0\,.
\end{equation*}
Now, we define
\begin{equation*}
\rho_\lambda = \left( \frac{N-1}{\lambda}\right)^{\frac{1}{1-q}} \,.
\end{equation*}
Thus, if $r \le \rho_\lambda$, then $h'(r)<0$ may hold, and if $r>\rho_\lambda$ the solution must satisfy $h'(r)=0$.

We assume $0<\rho_\lambda<R$. Then, when $\rho_\lambda \le r\le R$ the solution to problem \eqref{ejemplo-1} is constant, and since we know that $h(R)=0$ we deduce that $h(r)=0$ for all $\rho_\lambda \le r\le R$.

On account of \eqref{dinar}, if $0\le r < \rho_\lambda$, then solution is given by
\begin{multline*}
h(\rho_\lambda)-h(r)=\int_r^{\rho_\lambda} h'(s)\,ds= \int_r^{\rho_\lambda} \left(\frac{N-1}{s}-\frac{\lambda}{s^q}\right) \,ds
\\= (N-1)\log\left(\frac{\rho_\lambda}{r}\right) +\frac{\lambda}{1-q}(r^{1-q}-\rho_\lambda^{1-q})\,.
\end{multline*}
Therefore,
\begin{equation*}
u(x)=\left\{
\begin{array}{lcc}
(N-1)\log\left(\dfrac{|x|}{\rho_\lambda}\right) +\dfrac{\lambda}{1-q}(\rho_\lambda^{1-q}-|x|^{1-q}) & \mbox{if} & 0\le |x| <\rho_\lambda \,, \\
0 & \mbox{if} & \rho_\lambda < |x| \le R \,.
\end{array}  \right.
\end{equation*}

The vector field $\z$ must be identified.
When $\rho_\lambda \le r\le R$ we know that the vector field is $\z(x)=-\frac{x}{|x|}$, and when $0\le r < \rho_\lambda$, we assume that the vector field is radial: $\z(x)=x\,\xi (|x|)$. Thus, $\Div \z(x)= N\,\xi(|x|)+|x|\,\xi'(|x|)$, and equation \eqref{cond-distribucion} becomes
\begin{equation*}
-(N\,\xi(r)+r\,\xi'(r)) = \frac{\lambda}{r^q} \,.
\end{equation*}
That is,
\begin{equation*}
-r^N\,\xi(r)=-\int \left(r^N\,\xi(r)\right)' \,dr = \int \lambda \, r^{N-1-q} \,dr = \frac{\lambda}{N-q}\, r^{N-q} +C \,,
\end{equation*}
for some constant $C$ to be determinate.
Then
\begin{equation*}
\xi(r) = -\frac{\lambda}{N-q}\, r^{-q} -C\, r^{-N} \,.
\end{equation*}
Since we need a continuous vector field and we know that $\z(x)=-\frac{x}{\rho_\lambda}$ for $x$ with $|x|=\rho_\lambda$, we get the following equation
\begin{equation*}
\rho_\lambda^{-1} =\frac{\lambda}{N-q}\, \rho_\lambda^{-q} +C\, \rho_\lambda^{-N} \,.
\end{equation*}
Finally, using that $\lambda = (N-1)\,\rho_\lambda^{q-1}$ we deduce
\begin{equation*}
C=\rho_\lambda^{N-1}\frac{1-q}{N-q} \,,
\end{equation*}
and therefore, the vector field is given by
\begin{equation*}
\z(x)=\left\{
\begin{array}{lcc}
-\dfrac{x}{|x|} & \mbox{if} & 0\le |x| <\rho_\lambda \,, \\
-\dfrac{x}{N-q}\left((N-1) \dfrac{\rho_\lambda^{q-1}}{|x|^{q}} +(1-q)\dfrac{\rho_\lambda^{N-1}}{|x|^N}\right) & \mbox{if} & \rho_\lambda < |x| \le R \,.
\end{array}
\right.
\end{equation*}

\begin{remark}\label{regul}
In our Theorem \ref{teo-regularidad} we have prove that if $f\in L^{\frac Nq}(B_R(0))$, then
 $u \in L^{\frac N{q-1}}(B_R(0))$.
Since $\frac{\lambda}{|x|^q} \in L^s(B_R(0))$ for all $s<\frac{N}{q}$, it follows that $u \in L^r(B_R(0))$ for all $r<\frac{N}{q-1}$. This is exactly what it is shown.
\end{remark}

\begin{remark}
In \cite[Proposition 4.4]{LS} it was proved that for any ``small'' datum $f \in W^{-1,\infty}(B_R(0))$, the solution to problem \eqref{prob-prin} is always trivial. Nevertheless, in our examples we always get a positive solution. This is due to the fact that the datum $f(x)=\lambda \,|x|^{-q}$ when $1<q<N$ is not in the space $W^{-1,\infty}(B_R(0))$:

Let $s=N-q$, then function $v(x)=|x|^{-s}-R^{-s} \in W^{1,1}_0(B_R(0))$ since $s<N-1$. However, the product $f(x)v(x)=\lambda|x|^{-N}-f(x)R^{-s}\not \in L^1(B_R(0))$. We conclude that $f \not \in W^{-1,\infty}(B_R(0))$.
\end{remark}

It may be worth comparing our example with that occurring when the datum is $\frac{\lambda}{|x|^q}$, with $0<q<1$. In the same way as in Example \ref{ej-1}, the solution to problem \eqref{ejemplo-1} depends on a value
\begin{equation*}
r_\lambda = \left(\frac{N-q}{\lambda}\right)^{\frac{1}{1-q}} \,.
\end{equation*}
When $0<q<1$, the solution to problem \eqref{ejemplo-1} is given by
\begin{equation*}
u(x)=\left\{
\begin{array}{lcc}
(N-1) \log\left(\dfrac{r_\lambda}{R}\right)+\dfrac{\lambda}{1-q}(R^{1-q}-r_\lambda^{1-q}) & \mbox{if} & 0 \le |x| \le r_\lambda \,, \\
(N-1)\log\left(\dfrac{|x|}{R}\right) +\dfrac{\lambda}{1-q}(R^{1-q}-|x|^{1-q}) & \mbox{if} & r_\lambda< |x| \le R \,,
\end{array}  \right.
\end{equation*}
and the vector field associated is
\begin{equation*}
\z(x)=\left\{
\begin{array}{lcc}
-\dfrac{\lambda}{N-q}x|x|^{-q} & \mbox{if} & 0\le |x| \le r_\lambda \,, \\
-\dfrac{x}{|x|} & \mbox{if} & r_\lambda < |x| \le R \,.
\end{array}
\right.
\end{equation*}
It is easy to see that, since $0<q<1$, this solution is always bounded.

\end{document}